\documentclass[reqno,b5paper]{amsart}
\usepackage{amsmath}
\usepackage{amssymb}
\usepackage{amsthm}
\usepackage{enumerate}
\usepackage[mathscr]{eucal}

\newtheorem{lemma}{Lemma}[section]
\newtheorem{definition}{Definition}[section]
\newtheorem{example}{Example}[section]

\newtheorem{theorem}{Theorem}[section]

\newtheorem{remark}{Remark}[section]

\begin{document}

\title[Relations between minimal usco and minimal cusco  maps]{Relations between minimal usco and minimal cusco  maps}
\author[\v Lubica Hol\'a \and Du\v san Hol\'y]{\v Lubica Hol\'a* \and Du\v san Hol\'y**}

\newcommand{\acr}{\newline\indent}

\address{\llap{*}Academy of Sciences, Institute of Mathematics \acr \v
Stef\'anikova 49,
81473 Bratislava, Slovakia
\acr Slovakia}

\email{hola@mat.savba.sk}

\address{\llap{**\,} Department of Mathematics and Computer Science, Faculty of Education, Trnava University, \acr
Priemyseln\'a 4,
918 43 Trnava, Slovakia
\acr Slovakia}

\email{dusan.holy@truni.sk}

\thanks{}

\keywords{minimal cusco map, minimal usco map, quasicontinuous function, subcontinuous function, set-valued mapping, selection, extreme function.
Both authors would like to thank to grant APVV-0269-11 and \v L. Hol\'a also to grant Vega 2/0018/13.}

\begin{abstract}
In our paper we give a characterization of (set-valued) maps  which are minimal usco and minimal cusco simultaneously. Let $X$ be a topological space and $Y$ be a Banach space. We  show that there is a bijection  between the space $MU(X,Y)$ of minimal usco maps from $X$ to $Y$   and the space $MC(X,Y)$ of minimal cusco maps from $X$ to $Y$ and we study this bijection with respect to various topologies on underlying spaces. Let $X$ be a Baire space and $Y$ be a Banach space. Then $(MU(X,Y),\tau_U)$ and $(MC(X,Y),\tau_U)$ are homeomorphic, where $\tau_U$ is the topology of uniform convergence.
\end{abstract}

\maketitle

\section{Introduction}

\bigskip

\bigskip

The acronym usco (cusco) stands for a (convex) upper semicontinuous non-empty compact-valued set-valued map. Such set-valued maps are interesting because they describe common features of maximal monotone operators, of the convex subdifferential and of Clarke generalized gradient. Examination of cuscos and uscos leads to serious insights into the underlying topological properties of the convex subdifferential and the Clarke generalized gradient. (It is known that Clarke subdifferential of a locally Lipschitz function and, in particular, the subdifferential of a convex continuous functions are weak* cuscos.) (see [BZ1])

In our paper we are interested in minimal usco and minimal cusco maps. We give a characterization of such maps which are minimal usco and minimal cusco simultaneously. We also show that there is a bijection between the space of minimal usco maps  and the space of minimal cusco maps and we study this bijection under the topologies of pointwise convergence, uniform convergence on compacta, uniform convergence and under the Vietoris topology.

Minimal usco and minimal cusco maps are used in many papers (see [BZ1], [BZ2], [DL], [GM], [HH1], [HH3], [HoM], [Wa]). Historically, minimal usco maps seem to have appeared first in complex analysis (in the second half of the 19th century), in the form of a bounded holomorphic function and its "cluster sets", see e.g. [CL]. Minimal usco maps are a very convenient tool in the theory of games (see [Ch]) or in functional analysis (see [BM]), where a differentiability property of single-valued functions is characterized by their Clarke subdifferentials being minimal cuscos.

\rm

\bigskip
\section{Minimal usco and minimal cusco maps}
\bigskip
In what follows let $X, Y$ be Hausdorff topological spaces, $\Bbb R$ be the space of real numbers with the usual metric and $Z^+$ be the set of positive integers. Also, for $x \in X$, $\mathcal U(x)$ is always used to denote a base of open neighborhoods of $x$ in $X$. The symbol $\overline A$ and $Int A$ will stand for the closure and interior of the set $A$ in a topological space.

A set-valued map, or a multifunction, from $X$ to $Y$ is a function that assigns to each element of $X$ a subset of $Y$. If $F$ is a set-valued map from $X$ to $Y$, then its graph is the set  $\{(x,y) \in X \times Y: y \in F(x)\}$. Conversely, if $F$ is a subset of $X \times Y$ and $x \in X$, define $F(x) = \{y \in Y: (x,y) \in F\}$. Then we can assign to each subset $F$ of $X \times Y$ a set-valued map which takes the value $F(x)$ at each point $x \in X$ and which graph is $F$. In this way, we identify set-valued maps with their graphs. Following [DL] the term map is reserved for a set-valued map.

Notice that if $f: X \to Y$  is a single-valued function, we will use the symbol $f$ also for the graph of $f$.

\bigskip
Given two maps $F, G: X \to Y$, we write $G \subset F$ and say that $G$ is contained in $F$ if $G(x) \subset F(x)$ for every $x \in X$.

A map $F: X \to Y$ is upper semicontinuous at a point $x \in X$ if for every open set $V$ containing $F(x)$, there exists $U \in \mathcal U(x)$ such that
\bigskip

\centerline{$F(U) = \cup \{F(u):u \in U\} \subset V.$}
\bigskip

$F$ is upper semicontinuous if it is upper semicontinuous at each point of $X$. Following Christensen [Ch] we say, that a map $F$ is usco if it is upper semicontinuous and takes nonempty compact values. A map $F$ from a topological space $X$ to a linear topological space $Y$ is cusco if it is usco and $F(x)$ is convex for every $x \in X$.

Finally, a map $F$ from a topological space $X$ to a topological (linear topological space) $Y$ is said to be minimal usco (minimal cusco) if it is a minimal element in the family of all usco (cusco) maps (with domain $X$ and range $Y$); that is, if it is usco (cusco) and does not contain properly any other usco (cusco) map from $X$ into $Y$. By an easy application of the Kuratowski-Zorn principle we can guarantee that every usco (cusco)  map from $X$ to $Y$ contains a minimal usco (cusco) map from $X$ to $Y$ (see [BZ1], [BZ2], [DL]).

Other approach to minimality of set-valued maps can be found in [Ma] and [KKM].

In the paper [HH1] we can find an interesting characterization of minimal usco maps using quasicontinuous and subcontinuous selections, which will be also useful for our analysis.

\bigskip

A function $f: X \to Y$ is quasicontinuous at $x \in X$ [Ne] if for every neighborhood $V$ of $f(x)$ and every $U \in \mathcal U(x)$ there is a nonempty open set $G \subset U$ such that $f(G) \subset V$. If $f$ is quasicontinuous at every point of $X$, we say that $f$ is quasicontinuous.

The notion of quasicontinuity was perhaps the first time used by R. Baire in [Ba] in the study of points of separately continuous functions. As Baire indicated in his paper [Ba] the condition of quasicontinuity has been suggested by Vito Volterra. There is a rich literature concerning the study of quasicontinuity, see for example [Ba], [HP], [Ke], [KKM], [Ne]. A condition under which the pointwise limit of a sequence of quasicontinuous functions is quasicontinuous was studied in [HH2].

A function $f: X  \to Y$ is subcontinuous at $x \in X$ [Fu] if for every net $(x_i)$ convergent to $x$, there is a convergent subnet of $(f(x_i))$. If $f$ is subcontinuous at every $x \in X$, we say that $f$ is subcontinuous.

Let $F:X \to Y$ be a set-valued map. Then a function $f:X \to Y$ is called a selection of $F$ if $f(x) \in F(x)$ for every $x \in X$.

It is well known that every selection of a usco map is subcontinuous ([HH1], [HN]).

\begin{theorem} (see Theorem 2.5 in [HH1])  Let $X, Y$ be topological spaces and $Y$ be a $T_1$ regular space. Let $F$ be a map from $X$ to $Y$. The following are equivalent:

(1) $F$ is a minimal usco map;

(2) There exist a quasicontinuous and subcontinuous selection $f$ of $F$ such that $\overline f = F$;

(3) Every selection $f$ of $F$ is quasicontinuous, subcontinuous and $\overline f = F$.
\end{theorem}

\bigskip

Let $Y$ be a linear topological space and $B \subset Y$ is a set. By $\overline{co} B$ we denote the closed convex hull of the set $B$ (see [AB]).

\bigskip

The following Lemma is a folklore.

\begin{lemma} Let $X$ be a topological space and $Y$ be a Hausdorff locally convex linear topological space. Let $G$ be a usco map from $X$ to $Y$ and  $\overline{co}$$ G(x)$ is compact for every $x \in X$. Then the map $F$ defined as $F(x) = \overline{co}$$ G(x)$ for every $x \in X$ is a cusco map.
\end{lemma}

\bigskip

\begin{remark} There are three important cases when the closed convex hull of a compact set is compact. The first is when the compact set is a finite union of compact convex sets. The second is when the space is completely metrizable and locally convex. This includes the case of all Banach spaces with their norm topologies. The third case is a compact set in the weak topology on a Banach space. (see[AB])
\end{remark}

Let $B$ be a subset of a linear topological space. By $\mathcal E(B)$ we denote the set of all extreme points of $B$.
Let $X$ be a topological space and $Y$ be a Hausdorff locally convex (linear topological) space. Let $F: X \to Y$ be a map with nonempty compact values. Then a selection $f$ of $F$ such that $f(x) \in \mathcal E(F(x))$ for every $x \in X$ is called an extreme function of $F$.

\bigskip

Notice that all known characterizations of minimal cusco maps are given in the class of cusco maps (see [GM], [BZ1]). So the following characterization of minimal cusco maps in the class of all set-valued maps can be of some interest:

\begin{theorem} (see [HH3]) Let $X$ be a topological space and $Y$ be a Hausdorff locally convex (linear topological) space. Let $F$ be a map from $X$ to $Y$. Then the following are equivalent:

(1) $F$ is a minimal cusco map;

(2) $F$ has nonempty compact values and there is a quasicontinuous, subcontinuous selection $f$ of $F$ such that $\overline {co}$$ \overline f(x) = F(x)$ for every $x \in X$;

(3) $F$ has nonempty compact, convex values, $F$ has a closed graph, every extreme function of $F$ is quasicontinuous, subcontinuous and any two extreme functions of $F$ have the same closures of their graphs;

(4) $F$ has nonempty compact values, every extreme function $f$ of $F$ is quasicontinuous, subcontinuous   and $F(x) = \overline{co}$$ \overline f(x)$ for every $x \in X$.

\end{theorem}

\bigskip

Let $X$ be a topological space and $Y$ be a Hausdorff locally convex (linear topological) space. Denote by $MU(X,Y)$ the set of all minimal usco maps and by $MC(X,Y)$
the set of all minimal cusco maps from $X$ to $Y$. Of course
$MU(X,Y)\cap MC(X,Y)\neq\emptyset$.  It follows from the next example that
$MU(X,Y)\setminus MC(X,Y)\neq\emptyset$ and also $MC(X,Y)\setminus
MU(X,Y)\neq\emptyset$.

\bigskip

\begin{example}Let $X = [-1,1]$ with the usual Euclidean
topology. Consider the maps $F$ and $G$ from $X$ to $\Bbb R$ defined
by:

\bigskip

\begin{center}
$F(x)=\left\{
  \begin{array}{ll}
    1, & \hbox{$x \in [-1,0)$;} \\
    \{-1,1\}, & \hbox{$x=0$;} \\
    -1, & \hbox{$x\in (0,1]$.}
  \end{array}
\right.$

\bigskip

$G(x)=\left\{
  \begin{array}{ll}
    1, & \hbox{$x \in [-1,0)$;} \\
   \hbox{$[-1,1]$}, & \hbox{$x=0$;} \\
    -1, & \hbox{$x\in (0,1]$.}
  \end{array}
\right.$
\end{center}

\end{example}

\bigskip

\begin{definition}Let $X$ be a topological space and $Y$ be a
Hausdorff locally convex (linear topological) space. We say that $f$
is *-quasicontinuous at $x$ if for
every $y\in\overline{co}$$\overline f(x)$,
for every $V \in \mathcal U(y)$ and every $U \in \mathcal U(x)$ there is a
nonempty open set $W \subset U$ such that $f(W) \subset V$. If $f$
is *-quasicontinuous at every point of $X$, we say that $f$ is
*-quasicontinuous.
\end{definition}

\bigskip

\begin{example} Consider the function $f$ from $\Bbb R$ to $\Bbb
R$ defined by:

\bigskip

\begin{center}
$f(x)=\left\{
  \begin{array}{ll}
    \sin\frac {1}{x}, & \hbox{$x\neq 0$;} \\
    0 & \hbox{$x=0$.}
  \end{array}
\right.$
\end{center}

The function $f$ is not continuous at $x=0$, but it is *-quasicontinuous
at $0$.
\end{example}

\begin{theorem}Let $X$ be a topological space and $Y$ be a
Hausdorff locally convex (linear topological) space. Let $f$ be a function from $X$ to $Y$. Then
$f$ is
 *-quasicontinuous  at $x$ if and only if every selection of
 $\overline f$ is quasicontinuous at $x$ and $\overline
{f}(x)=\overline{co}$$\overline f(x)$.
\end{theorem}

It follows from the previous Theorem  that  every *-quasicontinuous function is quasicontinuous.

\bigskip

\begin{theorem} Let $X$ be a topological space and $Y$ be a
Hausdorff locally convex (linear topological) space. Let $F$ be a map
from $X$ to $Y$. Then the following are equivalent:

(1) $F\in MU(X,Y)\cap MC(X,Y)$;

(2) There exist a *-quasicontinuous and subcontinuous function $f$ from $X$ to $Y$
such that $\overline {f}= F$;

(3) Every selection $f$ of $F$ is
*-quasicontinuous, subcontinuous and $\overline {f}= F$.
\end{theorem}

\begin{proof}$(1) \Rightarrow (3)$ Let $f$ be a selection of $F$.
Since $F\in MU(X,Y)$, by Theorem 2.1 $f$ is
quasicontinuous, subcontinuous and $\overline {f}= F$. Since $F\in MC(X,Y)$,  $\overline {co}$$ \overline f(x) = F(x)$. So by Theorem 2.3 $f$ is
*-quasicontinuous.

$(3) \Rightarrow (2)$ is trivial.

$(2) \Rightarrow (1)$ Let $f$ be a *-quasicontinuous and
subcontinuous function from $X$ to $Y$ such that $\overline
{f}= F$. Thus $f$ is quasicontinuous and by Theorem 2.3 $F(x)=\overline
{f}(x)=\overline{co}$$\overline f(x)$. By Theorem 2.1 $F$
is minimal usco map and since by Lemma 2.1 $F$ is cusco,  by Proposition 2.7 in [BZ1] $F$ is minimal cusco.
\end{proof}

\bigskip

Denote by $F(X,Y)$ the set of all maps with nonempty closed values from a topological space $X$ to a Hausdorff locally convex (linear topological) space $Y$. Define the function $\varphi :\ MU(X,Y)\to F(X,Y)$ as follows: $\varphi (F)(x)=\overline{co}$$ F(x)$.

\bigskip

The following theorem was proved in [HH3]. For a reader's convenience we will also  provide the proof.

\begin{theorem} (see Theorem 4.1 in [HH3]) Let $X$ be a topological space and $Y$ be a Hausdorff locally convex (linear topological) space. Let $F: X \to Y$ be a minimal cusco map. There is a unique minimal usco map contained in $F$.
\end{theorem}
\begin{proof}  Let $G, H$ be two minimal usco maps contained in $F$. It is sufficient to prove that $G(x) \cap H(x) \ne \emptyset$ for every $x \in X$.
Then a map $L: X \to Y$ defined as $L(x) = G(x) \cap H(x)$ for every $x \in X$ is usco and $L \subset G$, $L \subset H$. Thus $G = L = H$.

Let $\tau$ be a Hausdorff locally convex topology on $Y$ and $\Gamma$ be a system of seminorms on $X$ which generate $\tau$. For every $x \in X$, every $p \in \Gamma$ and $\epsilon > 0$ we denote

\bigskip
\centerline{$S_{p,\epsilon}(x) = \{y \in Y: p(x - y) < \epsilon\}$ and $S_{p,\epsilon}(A) = \cup_{a \in A} S_{p,\epsilon}(a)$.}
\bigskip

Suppose there is $x \in X$ such that $G(x) \cap H(x) = \emptyset$. Since $G(x), H(x)$ are compact sets, there is a seminorm $p$ and $\epsilon > 0$ such that
\bigskip

\centerline{$S_{p,\epsilon}(G(x)) \cap S_{p,\epsilon}(H(x)) = \emptyset.$}
\bigskip

The upper semicontinuity of $G$ and $H$ implies that there is $U \in \mathcal U(x)$ such that

\bigskip
\centerline{$G(z) \subset S_{p,\epsilon}(G(x))$ and $H(z) \subset S_{p,\epsilon}(H(x))$}
\bigskip
for every $z \in U$. Let $g \subset G$ and $h \subset H$ be selections of $G, H$ respectively. The quasicontinuity of $h$ at $x$ implies that there is a nonempty open set $V \subset U$ such that

\bigskip
\centerline{$h(V) \subset S_{p,\epsilon/2}(h(x)) \subset \overline{S_{p,\epsilon/2}(h(x))} \subset S_{p,\epsilon}(h(x))$.}

\bigskip
Thus $\overline h(V) \subset \overline{S_{p,\epsilon/2}(h(x))}$, i.e. $\overline{co}$$\overline h(z) \subset \overline{S_{p,\epsilon/2}(h(x))} \subset S_{p,\epsilon}(h(x)) $ for every $z \in V$. For every $z \in V$ we have $F(z) = \overline{co}$$\overline h(z) \subset S_{p,\epsilon}(h(x))$. Since $g(z) \in F(z)$, we have $g(z) \in S_{p,\epsilon}(h(x))$, a contradiction.
\end{proof}

\begin{theorem}Let $X$ be a topological space and $Y$ be a Hausdorff locally convex (linear topological) space in which the closed convex hull of a compact set is compact. The map $\varphi$ is bijection from $MU(X,Y)$ to $MC(X,Y)$.
\end{theorem}
\begin{proof}Let $F\in MU(X,Y)$.  To show that $\varphi (F)\in MC(X,Y)$ note that by Lemma 2.1 the map $G$ defined as $G(x) = \overline{co}F(x)$ for every $x \in X$ is a cusco map and by Proposition 2.7 in [BZ1]  $G$ is minimal cusco.

Next we show that $\varphi$ maps $MU(X,Y)$ onto $MC(X,Y)$. Let $G\in MC(X,Y)$ and let $F$ be a minimal usco map contained in $G$. By Lemma 2.1 the map $x\rightarrow \overline{co}F(x)$ is a cusco map such that $\overline{co}F(x)\subset G(x)$ for every $x \in X$. Since $G$ is minimal cusco, $G(x) = \overline{co}F(x)$ for every $x \in X$.

Finally,  to show that $\varphi$ is one-to-one, suppose that $F,\
G\in MU(X,Y)$ and $F\neq G$. Suppose, by way of contradiction, that $\varphi (F)=\varphi (G)$. So by Theorem 2.5 $F=G$, a contradiction.

\end{proof}
\bigskip
\section{Topological properties of $\varphi$}

Let $X$ be a Hausdorff topological space and $(Y,d)$ be a metric space.
The open $d$-ball with center $z_0\in Y$ and radius $\varepsilon
>0$ will be denoted by $S_\varepsilon (z_0)$ and the
$\varepsilon $-paralel body $\bigcup_{a\in A}S_\varepsilon (a)$ for
a subset $A$ of $Y$ will be denoted by $S_\varepsilon
(A)$.\newline
Denote by $CL(Y)$ the space of all nonempty closed subsets of $Y$. By $\frak K (X)$ and $\frak F(X)$ we
mean the family of all nonempty compact and finite subsets of $X$, respectively.
If $A\in CL(Y)$, the distance functional $d(.,A)\ :Y
\mapsto[0,\infty )$ is described by the familiar formula
$$d(z,A)=\inf \{d(z,a):\ a\in A\}.$$

Let $A$ and $B$ be nonempty subsets of $(Y,d)$. The excess of $A$ over $B$ with respect to $d$ is defined by the formula
$$e_d(A,B)=\sup\{d(a,B):\ a\in A\}.$$

The Hausdorff (extended-valued) metric $H_d$ on $CL(Y)$ [Be] is
defined by

$$H_d(A,B)=max\{e_d(A,B), e_d(B,A)\}.$$

 We will often use the following
equality on $CL(Y)$:
$$H_d(A,B)=inf\{\varepsilon > 0:\ A\subset S_\varepsilon (B)\ \text{and}\ B\subset
S_\varepsilon (A)\}.$$
The  topology generated by  $H_d$ is called the Hausdorff metric topology.

\bigskip

Following [HM] we will define the topology $\tau_p$ of
pointwise convergence on $F(X,Y)$. The topology $\tau_p$ of pointwise
convergence on $F(X,Y)$ is induced by the uniformity $\frak U_p$ of
pointwise convergence which has a base consisting of sets of the
form

$$W(A,\varepsilon )=\{(\Phi ,\Psi):\ \forall\ x\in A\ \ H_d(\Phi (x),\Psi(x))<
\varepsilon \},$$

where $\noindent A\in\frak F(X)$ and $\varepsilon >0$. The general
$\tau_p$-basic neighborhood of $\Phi\in F(X,Y)$ will be denoted by
$W(\Phi,A,\varepsilon )$, i.e. $W(\Phi,A,\varepsilon
)=W(A,\varepsilon )[\Phi] = \{\Psi:H_d(\Phi(x),\Psi(x)) <
\varepsilon$ for every $x \in A\}$. If $A=\{a\}$, we will write
$W(\Phi,a,\varepsilon )$ instead of $W(\Phi ,\{a\},\varepsilon )$.

\bigskip

We will define the topology $\tau_{UC}$ of uniform convergence on
compact sets on $F(X,Y)$ [HM]. This topology is induced by the
uniformity $\frak U_{UC}$  which has a base consisting of sets of the
form

$$W(K,\varepsilon )=\{(\Phi ,\Psi):\ \forall\ x\in K\ \ H_d(\Phi (x),\Psi(x))<
\varepsilon \},$$

where  $\noindent K\in\frak K(X)$ and $\varepsilon >0$. The general
$\tau_{UC}$-basic neighborhood of $\Phi\in F(X,Y)$ will be denoted by
$W(\Phi,K,\varepsilon )$, i.e. $W(\Phi,K,\varepsilon
)=W(K,\varepsilon )[\Phi]$.

\bigskip

Finally we will define the topology $\tau_U$ of uniform
convergence on $F(X,Y)$  [HM].  Let $\varrho$ be the (extended-valued)
metric on $F(X,Y)$ defined by
$$\varrho(\Phi,\Psi )=\sup\{H_d(\Phi (x),\Psi (x)):\ x\in X\},$$ for each
$\Phi ,\Psi \in F(X,Y)$. Then the topology of uniform convergence for
the space $F(X,Y)$ is the topology generated by the metric $\varrho$.

\bigskip

\begin{lemma} Let $Y$ be a normed linear space. Let $A,B$ be nonempty closed subsets of $Y$. Then $H_d(\overline{co} A,\overline{co} B)\leq H_d(A,B)$.

\end{lemma}
\begin{proof}
At first we show that $e_d(\overline{co} A,\overline{co} B)\leq e_d(A,B)$. It is known (see [Be] exercise 1.5.3. b), that if $C$ is convex then $e_d(\overline{co} A, C)= e_d(A, C)$. So $e_d(\overline{co} A,\overline{co} B)= e_d(A,\overline{co} B)$. Since $B\subset\overline{co} B$ we have that
$$e_d(\overline{co} A,\overline{co} B)= e_d(A,\overline{co} B)\leq e_d(A,B).$$
Similarly we can show that
$$e_d(\overline{co} B,\overline{co} A)= e_d(B,\overline{co} A)\leq e_d(B,A).$$
Since for every $C,D\in CL(Y)$
$$H_d(C,D)=max\{e_d(C,D), e_d(D,C)\},$$ we are done.

\end{proof}

\bigskip

\begin{theorem}Let $X$ be a topological space and $Y$ be a Banach space. The map $\varphi$ from $(MU(X,Y),\tau )$ onto $(MC(X,Y),\tau )$ is continuous if $\tau$ is one of the following topologies $\tau_p, \tau_{UC},
\tau_U$.
\end{theorem}
\begin{proof}The proof follows from the
above Lemma.\end{proof}

\bigskip

The following example shows that the map $\varphi^{-1}$ from
$(MC([-1,1]),\tau_p )$ onto\newline $(MU([-1,1]),\tau_p )$ is not
continuous.

\begin{example}Let $X = [-1,1]$ with the usual Euclidean
topology. Let $F$ and $G$ are maps from Example 2.1. Then
$F=\varphi^{-1}(G)$. We claim that $\varphi^{-1}$ is not continuous
at $G$. For every $n\in\Bbb Z^+$ let $P_n$ be the
map from $[-1,1]$ to $\Bbb R$ defined by

\bigskip

\begin{center}

$P_n(x)=\left\{
  \begin{array}{ll}
    1, & \hbox{$x\in [-1,0)$;} \\
    \hbox{$[-1,1]$}, & \hbox{$x=0$;} \\
    \hbox{$\sin\frac {1}{x}$}, & \hbox{$x\in (0,\frac 2{(4n-1)\pi}]$;} \\
    -1, & \hbox{$x\in [\frac {2}{(4n-1)\pi},1]$.}
  \end{array}
\right.$

\end{center}

\bigskip

 It is easy to see that
for every $\noindent A\in\frak F(X)$ and every $\varepsilon >0$
there exists an $n_0\in\Bbb Z^+$ such that $P_n\in W(G,A,\varepsilon
)$ for every $n\geq n_0$. For every $n\in\Bbb Z^+$ we have that
$H_d(F(0),\varphi^{-1}(P_n)(0))=1$. Then for every $n\in\Bbb Z^+$
$P_n\notin W(F,0,\frac 12)$ and so the map $\varphi^{-1}$ is not
continuous at $G$.
\end{example}

\bigskip

\begin{remark}
If $X$ is a Baire space, $Y$ is a normed linear space  and $F\in MU(X,Y)$, then there is a dense $G_\delta$-subset $E$ of $X$ such that $F(x)$ is single-valued for each $x\in E$. In fact, let $f: X \to Y$ be a quasicontinuous subcontinuous selection of $F$ such that $\overline f = F$ (see Theorem 2.1). By Theorem 4.1 in [HP] the set $C(f)$ of the points of continuity of $f$ is a dense $G_\delta$-set in $X$. It is easy to verify that $F$ must be single-valued at every point of the set $C(f)$.
The same holds also for $F\in MC(X,Y)$.
\end{remark}

\bigskip

\begin{theorem}Let $X$ be a locally compact space and $Y$ be a Banach space. The map
$\varphi$ from $(MU(X,Y),\tau_{UC})$ onto $(MC(X,Y),\tau_{UC})$ is
homeomorphism.\end{theorem}
\begin{proof} By Theorem 2.6  $\varphi$ is a bijection. By Theorem 3.1 it is sufficient to prove that
$\varphi^{-1}$ is continuous. Let $G\in MC(X,Y)$ and
$F=\varphi^{-1}(G)$. Let $K\in\frak K(X)$ and $\varepsilon
>0$. We show that there exist $K_1\in\frak K(X)$ and
$\varepsilon_1 >0$ such that $\varphi^{-1}(W(G,K_1,\varepsilon_1
))\subset W(F,K,\varepsilon )$. Let $K_1\in\frak K(X)$ be such that
$K\subset IntK_1$. Put $\varepsilon_1=\frac\varepsilon 3$. Let $H\in
W(G,K_1,\varepsilon_1 )$ and $x\in K$. We show that $F(x)\subset
S_\varepsilon (\varphi^{-1}(H)(x))$. Let $y\in F(x)$. By Remark 3.1 and Theorem 2.1 for
$\frac\varepsilon 3$ and for every $U\in\mathcal U(x)$ there exist
$x_U\in U\cap IntK_1$ such that $F(x_U)$ is single-valued and $F(x_U)\in
S_{\frac\varepsilon 3}(y)$.  From the fact that $F(x_U)$ is single-valued it  follows that $G(x_U)$ is single-valued too and consequently $F(x_U)=G(x_U)$. Since $H(x_U)\subset S_{\frac\varepsilon 3}(G(x_U))$ we have that $\varphi^{-1}(H)(x_U)\subset S_{\frac\varepsilon 3}(F(x_U))$ and hence there exist $y_U\in \varphi^{-1}(H)(x_U)$ such that $d(y_U,F(x_U))<\frac\varepsilon 3$. Hence $d(y,y_U) < \frac{2\varepsilon} 3$. So there exists a subnet of the net
$\{(x_U,y_U):\ U\in\mathcal U(x)\}$ which converges
to a point $(x,z)$, where $z\in\varphi^{-1}(H)(x)$. So $F(x)\subset
S_\varepsilon (\varphi^{-1}(H)(x))$. The inclusion
$\varphi^{-1}(H)(x)\subset S_\varepsilon (F(x))$ can be  proved similarly.
\end{proof}

The following Example shows that the condition of local compactness in Theorem 3.2 is essential.

\begin{example} Let $X = [-1,1]$ with the topology, where the open sets in $X$ are all subsets of $X$ not containing $0$ and all subsets of $X$ containing $0$ that have countable complement. Every compact set in $X$ is finite. Thus the topology $\tau_{UC}$ on $MU(X,\Bbb R)$ and $MC(X,\Bbb R)$ reduces to the topology $\tau_p$. So we can use Example 3.1.

\end{example}

\begin{theorem}Let $X$ be a Baire space and  $Y$ be a Banach space. The map $\varphi$ from $(MU(X,Y),\tau_U)$ onto
$(MC(X,Y),\tau_U)$ is homeomorphism.
\end{theorem}
\begin{proof}
The proof is similar to the proof of Theorem 3.2.
\end{proof}

\bigskip
\bigskip
\bigskip

In the last part of our paper we will consider the Vietoris topology $V$ on $MU(X,\Bbb R)$ and on $MC(X,\Bbb R)$. First we will consider the Vietoris topology $V$ on the space $CL(X\times\Bbb R)$ of nonempty closed subsets of $X \times \Bbb R$. The basic open subsets of $CL(X \times \Bbb R)$ in $V$
are the subsets of the form

\bigskip

\centerline{$W^+ \cap W_1^- \cap ... \cap W_n^-$,}
\bigskip
where $W, W_1, ..., W_n$ are open subsets of $X \times R$, $W^+ =
\{F \in CL(X \times \Bbb R): F \subset W\}$, and each $W_i^- = \{F \in CL(X \times \Bbb R): F
\cap W_i \ne \emptyset\}$.

\bigskip
Under the identification of every element of $MU(X \times \Bbb R)$ and $MC(X \times \Bbb R)$ with its graph, we can consider $MU(X \times \Bbb R)$ and  $MC(X \times \Bbb R)$ as subsets of $CL(X \times \Bbb R)$. We will consider the induced Vietoris topology $V$ on $MU(X \times \Bbb R)$ and on  $MC(X \times \Bbb R)$.

\bigskip

\begin{theorem} Let $X$ be a locally connected space. The map
$\varphi$ from $(MU(X,\Bbb R),V)$ onto $(MC(X,\Bbb R),V)$ is continuous.

\end{theorem}

\begin{proof}Let $F\in MU(X,\Bbb R)$ and $W^+ \cap W_1^- \cap ... \cap
W_n^-$ be a basic open set in $(MC(X,\Bbb R),V)$ such that $\varphi(F)\in
W^+ \cap W_1^- \cap ... \cap W_n^-$.

Let $G=\varphi(F)$. By Lemma 4.1 in [HJM] there is an open set
$H \subset X \times \Bbb R$ such that $G \subset H \subset W$ and $H(x)$
is connected for every $x \in X$. Without loss of generality we can
also suppose that for every $i = 1, 2, ... n$,  $W_i \subset H$  and
$W_i = U_i \times V_i$, $U_i$ open in $X$, $V_i$ an open interval in
$\Bbb R$.

For every $i \in \{1,2, ... n\}$ we will define an open set $\mathcal H_i$
as follows. Let $i \in \{1, 2, ... n\}$. Let $(x_i,y_i) \in G \cap
W_i$. If $y_i = inf F(x_i)$ or $y_i = sup F(x_i)$, we will put
$\mathcal H_i = W_i^-$. Otherwise, let $C_i$ be a connected set in $X$ such
that $x_i \in Int C_i \subset C_i \subset U_i$ and $\epsilon > 0$ be
such that $inf F(x_i) + \epsilon < y_i < sup F(x_i) - \epsilon$. Put
$\mathcal H_i =$
\bigskip
$(Int C_i \times (inf F(x_i) - \epsilon,inf F(x_i) + \epsilon))^-
\cap (Int C_i \times (sup F(x_i) - \epsilon,sup F(x_i) +
\epsilon))^-$.

It is easy to verify that $L \in MU(X,\Bbb R) \cap \mathcal H_i$ implies that
$\varphi(L) \in \mathcal H_i$. Since $\varphi(L)$ is upper semicontinuous
set-valued map with connected values, $\varphi(L)(Int C_i)$ must be
a connected set ([Be, Proposition 6.2.12]); i.e. $y_i \in \varphi(L)(Int C_i)$. Thus
$\varphi(L) \in  \cap W_i^-$.

Put $\mathcal G = H^+ \cap \mathcal H_1 \cap ... \cap \mathcal H_n$. Then $F \in \mathcal
G$ and $\varphi(S) \in W^+ \cap W_1^- \cap ... \cap W_n^-$ for every
$S \in \mathcal G$.

\end{proof}

The following Example shows that the condition of local connectedness in the above Theorem is essential.

\begin{example}Let $X = [-1,1]\setminus\{\frac 1n:\ n\in\Bbb
Z^+\}$ with the usual Euclidean topology. Consider the maps $F$ and
$G$ from $X$ to $\Bbb R$ defined by

\begin{center}
$F(x)=\left\{
  \begin{array}{ll}
    1, & \hbox{$x\in X\cap [-1,0)$;} \\
    \{-1,1\}, & \hbox{$x=0$;} \\
    -1, & \hbox{$x\in X\cap (0,1]$.}
  \end{array}
\right.$

\bigskip

$G(x)=\left\{
  \begin{array}{ll}
    1, & \hbox{$x\in X\cap [-1,0)$;} \\
   \hbox{$[-1,1]$}, & \hbox{$x=0$;} \\
    -1, & \hbox{$x\in X\cap (0,1]$.}
  \end{array}
\right.$
\end{center}

 Then $G=\varphi (F)$ and we claim
that $\varphi$ is not continuous at $F$. For every $n \in \Bbb Z^+$ let $f_n$ be the function from $X$ to $\Bbb R$ defined by

\begin{center}
$f_n(x)=\left\{
  \begin{array}{ll}
    1, & \hbox{$x\in X\cap [-1,\frac 1n)$;} \\
    -1, & \hbox{$x\in X\cap (\frac 1n,1]$.}
  \end{array}
\right.$
\end{center}

We have that $f_n=\varphi(f_n)$ for every $n\in\Bbb Z^+$. The
sequence $\{f_n :\ n\in\Bbb Z^+\}$ converges in $(MU(X,\Bbb R),V)$ to $F$,
but $\{f_n :\ n\in\Bbb Z^+\}$ does not converge to $G$ in
$(MC(X,\Bbb R),V)$.
\end{example}

The following example shows that the map $\varphi^{-1}$ from
$(MC([-1,1],\Bbb R),V )$ onto $(MU([-1,1],\Bbb R),V )$ is not
continuous.

\begin{example}Let $X = [-1,1]$ with the usual Euclidean
topology. Let $F$, $G$ be maps from Example 2.2. Then
$F=\varphi^{-1}(G)$ and we claim that $\varphi^{-1}$ is not
continuous at $G$. For every $n \in \Bbb Z^+$ let $g_n$ be the
function from $[-1,1]$ to $\Bbb R$ defined by

\begin{center}

$g_n(x)=\left\{
  \begin{array}{ll}
    1, & \hbox{$x\in [-1,0]$;} \\
    1-2nx, & \hbox{$x\in (0,\frac 1n)$;} \\
    -1, & \hbox{$x\in [\frac 1n,1]$.}
  \end{array}
\right.$

\end{center}
Evidently $g_n=\varphi^{-1}(g_n)$ for every $n\in\Bbb
Z^+$. It is easy to see that the sequence $\{g_n :\ n\in\Bbb Z^+\}$
converges in $(MC([-1,1],\Bbb R),V)$ to $G$, but $\{g_n :\ n\in\Bbb Z^+\}$
does not converges to $F$ in $(MU([-1,1],\Bbb R),V)$.
\end{example}


\vskip 1pc

\begin{thebibliography}{xxxxx}



\bibitem[AB]{1}Aliprantis,~Ch.D.---Border,~K.C.: \textit{Infinite dimensional analysis}, 3rd edition, Springer Verlag Berlin, 2006

\bibitem[Ba]{2}Baire,~R.: \textit{Sur les fonctions des variables reelles}, Ann. Math. Pura Appl. \textbf{3} (1899), 1--122.

\bibitem[Be]{3}Beer,~G.: \textit{Topologies on closed and closed convex sets}, Kluwer Academic Publishers, 1993.

\bibitem[BM]{4}Borwein,~J.M.---Moors,~W.B.:\textit{Essentially smooth Lipschitz functions}, Journal of Functional Analyss \textbf{149} (1997), 305-351.

\bibitem[BZ1]{4}Borwein,~J.M.---Zhu,~Q.J.: \textit{Multifunctional and functional analytic techniques in nonsmooth analysis}, in Nonlinear Analysis, Differential Equations and Control, F.H. Clarke and R.J. Stern, editors, Kluwer Academic Publishers, 1999.

\bibitem[BZ2]{5}Borwein,~J.M.---Zhu,~Q.J.: \textit{Techniques of variational analysis}, Springer, 2005.

\bibitem[CL]{6}Collingwood,~E.F.---Lohwater,~A.J.: \textit{The theory of cluster sets}, Cambridge University Press, 1966.

\bibitem[Ch]{6}Christensen,~J.P.R.: \textit{Theorems of Namioka and R.E. Johnson type for upper semicontinuous and compact valued mappings}, Proc. Amer. Math. Soc. \textbf{86}  (1982), 649--655.

\bibitem[DL]{7}Drewnowski,~L.---Labuda,~I.: \textit{On minimal upper semicontinuous compact valued maps}, Rocky Mountain Journal of Mathematics \textbf{20} (1990), 737--752.

\bibitem[En]{8}Engelking,~R.: \textit{General Topology}, PWN 1977.

\bibitem[Fo]{9}Fort,~M.K.: \textit{Points of continuity of semi-continuous functions}, Publ. Math. Debrecen 2 1951-1952, 100--102.

\bibitem[Fu]{10}Fuller,~R.V.: \textit{Set of points of discontinuity}, Proc. Amer. Math. Soc. \textbf{38} (1973), 193--197.

\bibitem[GM]{11}Giles,~J.R.---Moors,~W.B.: \textit{A continuity property related to Kuratowski's index of non-compactness, its relevance to the drop property, and its implications for differentiability theory}, Journal of Mathematical Analysis and its Applications \textbf{178} (1993), 247--268.

\bibitem[HJM]{12}Hol\'a,~\v L.---Jain,~T.---McCoy,~R.A.: \textit{Topological properties of the multifunction space $L(X)$ of usco maps}, Mathematica Slovaca \textbf{58} (2008), 763-780.

\bibitem[HM]{12}Hammer,~S.T.---McCoy,~R.A.: \textit{Spaces of densely continuous forms}, Set-Valued Anal.  \textbf{5} (1997), 247--266.

\bibitem[HH1]{13}Hol\'a,~\v L.---Hol\'y,~D.: \textit{Minimal usco maps, densely continuous forms and upper semicontinuous functions}, Rocky Mountain Journal of Mathematics \textbf{39} (2009), 545--562.

\bibitem[HH2]{14}Hol\'a,~\v L.---Hol\'y,~D.: \textit{Pointwise convergence of quasicontinuous mappings and Baire spaces}, Rocky Mountain Journal of Mathematics \textbf{41}  (2011), 1883--1894.

\bibitem[HH3]{15}Hol\'a,~\v L.---Hol\'y,~D.: \textit{New characterization of minimal cusco maps}, Rocky Mountain Journal of Mathematics. accepted

\bibitem[HoM]{16}Hol\'y,~\v D.---Matej\'i\v cka,~L.: \textit{Quasicontinuous functions, minimal usco maps and topology of pointwise convergence}, Mathematica Slovaca \textbf{60} (2010), 506-520.

\bibitem[HN]{17}Hol\'a,~\v L.---Novotn\'y,~B.: \textit{Subcontinuity of multifunctions}, Mathematica Slovaca \textbf{62} (2012), 345-362.

\bibitem[HP]{18}Hol\'a,~\v L.---Piotrowski,~Z.: \textit{Set of continuity points of functions with values in generalized metric spaces}, Tatra Mountains Mathematical Publications \textbf{42}  (2009), 148--160.

\bibitem[Ke]{19}Kempisty,~S.: \textit{Sur les fonctions quasicontinues}, Fund. Math. \textbf{19} (1932),  189--197.

\bibitem[KKM]{20}Kenderov,~P.S.---Kortezov,~I.S.---Moors,~W.B.: \textit{Continuity points of quasi-continuous mappings}, Topology and its Applications, \textbf{109}  (2001), 321---346.

\bibitem[LL]{21}Lechicki,~A.---Levi,~S.: \textit{Extensions of semicontinuous multifunctions}, Forum Math. \textbf{2}  (1990), 341--360.


\bibitem[Ma]{22}Matejdes,~M.: \textit{Selection theorems and minimal mappings in a cluster setting}, Rocky Mountain Journal of Mathematics \textbf{41} (2011), 851-857.

\bibitem[Me]{23}Megginson,~R.E.: \textit{An introduction to Banach space theory}, Springer, 1998.

\bibitem[Mi]{24}Milman,~D.P.: \textit{Characteristics of extremal points of regularly convex sets}, Dokl. Akad. Nauk SSSR \textbf{57} (1947), 119-122.

\bibitem[Ne]{25}Neubrunn,~T.: \textit{Quasi-continuity} Real Anal. Exchange \textbf{14} (1998), 259--306.


\bibitem[No]{26}Novotn\'y,~B.: \textit{On subcontinuity} Real Anal. Exchange \textbf{31} (2005), 535-545.


\bibitem[Ph]{27}Phelps,~R. R.: \textit{Convex functions, monotone operators and differentiability}, Lecture Notes in Mathematics \textbf{1364}, Springer-Verlag, Berlin/Heidelberg/New York (1989).


\bibitem[Sp]{28}Spakowski,~A.: \textit{Upper set-convergences and minimal limits}, preprint.

\bibitem[Wa]{29}Wang,~X.: \textit{Asplund sets, differentiability and subdifferentiability of functions in Banach spaces}, Journal of Mathematical Analysis and Applications \textbf{323} (2006), 1417-1429.

\end{thebibliography}
\end{document}